\documentclass[11pt]{amsart}
\usepackage[T1]{fontenc}
\usepackage[english]{babel}
\usepackage{kpfonts}

\usepackage{amsmath,amssymb,amsthm,enumerate,xspace}
\usepackage{comment}
\usepackage[colorlinks=true,citecolor=blue,linkcolor=blue]{hyperref}
\usepackage{mathrsfs}

\numberwithin{equation}{section}

\newtheorem{thm}{Theorem}[section]
\newtheorem{prob}[thm]{Problem}

\newtheorem{question}[thm]{Question}
\newtheorem{prop}[thm]{Proposition}
\newtheorem{lem}[thm]{Lemma}
\newtheorem{cor}[thm]{Corollary}

\newtheorem*{iprob*}{Problem}

\theoremstyle{definition}

\newtheorem{rem}[thm]{Remark}

\newtheorem{exam}[thm]{Example}

\newcommand{\sK}{\mathscr{K}}
\newcommand{\sL}{\mathscr{L}}

\newcommand{\ZZ}{\mathbf{Z}}

\newcommand{\NN}{\mathbf{N}}

\newcommand{\se}{\subseteq}
\newcommand{\inv}{^{-1}}

\newcommand{\fhi}{\varphi}

\newcommand{\lra}{\longrightarrow}

\newcommand{\wt}{\widetilde}
\newcommand{\ul}{\underline}

\newcommand{\cat}[1]{{\upshape CAT({\ensuremath#1})}\xspace}

\DeclareMathOperator{\Alt}{Alt}
\DeclareMathOperator{\Sym}{Sym}
\DeclareMathOperator{\Supp}{Supp}
\newcommand{\means}{\mathscr{M}}
\newcommand{\Symf}{\Sym_\mathrm{f}}
\newcommand{\Altf}{\Alt_\mathrm{f}}
\newcommand{\paf}{\mathscr{P}_\mathrm{f}}
\newcommand{\pv}{\vee}
\newcommand{\bigpv}{\bigvee}
\newcommand{\lraa}{\lra\mathrel{\mkern-22mu}\lra} 
\newcommand{\tricycle}[3]{(\ul{#1}; \ul{#2}; \ul{#3})}
\newcommand{\abs}[1]{{\left| #1 \right|}}
\newcommand{\Born}{\mathfrak{B}}
\newcommand{\Power}{\mathscr{P}}
\DeclareMathOperator{\Fix}{Fix}
\usepackage{xcolor}

\title[Between free and direct products of groups]{Between free and direct products of groups}

\author[Maxime Gheysens]{Maxime Gheysens}
\address{TU Freiberg, Germany}
\author[Nicolas Monod]{Nicolas Monod}
\address{EPFL, Switzerland}
\begin{document}

\begin{abstract}
We investigate the group $G\pv H$ obtained by gluing together two groups $G$ and $H$ at the neutral element. This construction curiously shares some properties with the free product but others with the direct product.

Our results address among others Property~(T), \cat0 cubical complexes, local embeddability, amenable actions, and the algebraic structure of $G\pv H$.
\end{abstract}
\maketitle



\section{Introduction}

Given two groups $G$ and $H$, we can define an unorthodox sort of product group $G\pv H$ as follows. Take the disjoint union of $G$ and $H$ as sets and glue them together by identifying their neutral element $e$. On the resulting set, let $G$ act regularly on itself by left multiplication, and trivially elsewhere. Proceed similarly with $H$. Then $G\pv H$ is defined as the permutation group generated by these copies of $G$ and $H$.

This construction has a few quirks; for instance, if $G$ and $H$ are both finite, then $G\pv H$ is often a simple group (Theorem~\ref{thm:finite} below). Thus we shall focus mostly on infinite groups, where the following observation restores some credit to the concept of $G\pv H$ as a ``product'' of its subgroups $G, H$.

\begin{prop}
For infinite groups $G$ and $H$, there is a canonical epimorphism $G\pv H\twoheadrightarrow G\times H$ which is compatible with the inclusions of $G,H$ into $G\pv H$ and into $G\times H$.

Thus there is a canonical identification $(G\pv H)/[G, H] \cong G\times H$.
\end{prop}

In other words, for infinite groups, $G\pv H$ sits between the free product and the direct product with canonical epimorphisms
%
$$G * H \lraa G\pv H\lraa G\times H.$$
\itshape The theme of this article is that $G\pv H$ is similar to $G * H$ in some respects, but closer to $G\times H$ in others. \upshape Taken together, these antagonistic tendencies show that $G\pv H$ is a simple device for constructing unusual groups.

\medskip

A first elementary illustration of this duplexity is seen when comparing individual elements $g\in G$ and $h\in H$, viewed in $G\pv H$. In the free product, $g$ and $h$ would freely generate a free group as soon as they have infinite order. For $G\pv H$, one checks the same as long as inverses are not allowed:

\begin{prop}\label{prop:pong}
If $g\in G$ and $h\in H$ have infinite order, then they freely generate a free semigroup in $G\pv H$.
\end{prop}

In the direct product, much to the contrary, $g$ and $h$ commute. A simple computation shows that in $G\pv H$ the commutator $[g,h]$ is still trivial up to $3$-torsion:

\begin{prop}\label{prop:3tor}
For any $g\in G$ and $h\in H$, we have $[g,h]^3=e$ in $G\pv H$.
\end{prop}

For our next illustration, consider Kazhdan's property~(T). Recall that ${G\times H}$ has property~(T) if and only if both $G$ and $H$ do~\cite[1.9]{Harpe-Valette}. Contrariwise, $G*H$ never has property~(T) when $G$ and $H$ are non-trivial~\cite[6.a]{Harpe-Valette}.

\begin{thm}\label{thm:T}
Let $G$ and $H$ be infinite groups. Then $G\pv H$ does not have Kazhdan's property~(T).
\end{thm}

Thus, from the perspective of property~(T), it seems that $G\pv H$ is more similar to $G*H$ than to $G\times H$. A closer look, however, could support the opposite stance. Indeed, the proof that $G*H$ fails property~(T) comes from Bass--Serre theory. Namely, property~(T) implies the weaker property~(FA) of Serre, which is incompatible with free products. Here again, $G\times H$ has property~(FA) if and only if both $G$ and $H$ do. Now however $G\pv H$ shares this trait:

\begin{thm}
Let $G$ and $H$ be any groups. Then $G\pv H$ has Serre's property~(FA) if and only if both $G$ and $H$ do.
\end{thm}

From this it is clear that, contrary to the free product case, the obstruction to property~(T) recorded in Theorem~\ref{thm:T} does not come from an action on a tree: for instance, when $G$ and $H$ are infinite Kazhdan groups, $G\pv H$ fails~(T) but retains~(FA). The obstruction does, however, come from an infinite-dimensional generalisation of a tree:

\begin{thm}
Let $G$ and $H$ be infinite groups. Then $G\pv H$ acts by automorphisms on a \cat0 cubical complex $V$ without bounded orbits.

Moreover, each of $G$ and $H$ have fixed vertices in $V$, and there is a unique choice of such vertices that are adjacent.
\end{thm}

This complex can thus be seen as a (weak) replacement of the Bass--Serre tree for $G*H$.

\bigskip

Before going any further, we should lift some of the mystery around the structure of $G\pv H$. Since we defined $G\pv H$ as a permutation group generated by $G$ and $H$, we shall keep our notation straight by writing $\ul G$ and $\ul H$ for the sets with identified neutral elements; thus $\ul G \cap \ul H = \{\ul e\}$. Let further $\Altf$ denote the group of \emph{even} finitely supported permutations of any given set.

\begin{prop}\label{prop:E}
If $G$ and $H$ are non-trivial, then $G\pv H$ contains $\Altf(\ul G \cup \ul H)$.

Moreover, if $G$ and $H$ are infinite, then $\Altf(\ul G \cup \ul H)$ coincides with the kernel of the canonical epimorphism $G\pv H\twoheadrightarrow G\times H$.
\end{prop}

We should not conclude that all the mystery dissolves in view of the extension
$$1 \lra \Altf(\ul G \cup \ul H) \lra G\pv H \lra G\times H \lra 1.\leqno{\text{(E)}}$$
For instance, a general construction associates to any group $L$ its \textbf{lampshuffler} group $\Altf(L) \rtimes L$; however, $G\pv H$ cannot be described simply in terms of $G\times H$ and a lampshuffler. Beside the fact that $G\pv H$ is generated by just $G$ and $H$, the extension~(E) is more complicated than a semidirect product, as our next result shows. Recall that every infinite finitely generated group has either one, two or infinitely many ends and that the one-ended case is in a sense generic since the two others have strong structural restrictions by Stallings' theorem~\cite{Stallings68, Stallings_book}.

\begin{thm}
Let $G,H$ be one-ended finitely generated groups.

Then the canonical  epimorphism $G\pv H\twoheadrightarrow G\times H$ does not split.
\end{thm}

This contrasts with the fact that, by construction, each of the factors $G$ and $H$ lifts. It was pointed out to us by Yves Cornulier that the above statement can also be deduced from non-realisability results for near actions, specifically Theorem~7.C.1 in~\cite{Cornulier_near_arx} (see Proposition~2.6 in~\cite{Cornulier_real}). 

\medskip

An example of an issue that is immediately settled by the extension~(E) is the \emph{amenability} of $G\pv H$: this group is amenable if and only if both $G$ and $H$ are so. This is exactly like for $G\times H$, whereas the free product of infinite groups is never amenable.

But as soon as we refine the amenability question, $G\pv H$ swings back very close to $G*H$. Recall that a subgroup $G<\Pi$ is \textbf{co-amenable} in $\Pi$ if there is a $\Pi$-invariant mean on $\Pi/G$. Thus $G$ is co-amenable in $G\times H$ exactly when $H$ is amenable, but $G$ is only co-amenable in $G*H$ when $G*H$ itself is amenable or $H$ is trivial. The exact same behaviour is displayed by $G\pv H$:

\begin{thm}\label{thm:co-amen}
Given any two groups $G$ and $H$, the following are equivalent.

\begin{enumerate}[(i)]
\item $G$ is a co-amenable subgroup of $G\pv H$.
\item $G\pv H$ is amenable or $H$ is trivial.
\end{enumerate}
\end{thm}

We can also consider more general amenable actions of non-amenable groups. Van Douwen has initiated the study of the class of groups admitting a faithful transitive amenable action by showing that free groups belong to this class~\cite{vanDouwen}. Many more examples have been discovered; turning back to the free product $G*H$ of infinite groups, it always belongs to this class when \emph{at least one} of $G$ or $H$ does~\cite{Glasner-Monod}. In contrast, once again $G\times H$ belongs to that class if and only if both $G$ and $H$ do.

In the case of $G\pv H$, we have no definitive answer unless we specialise to \emph{doubly} transitive actions, in which case the next two propositions display once more the two-fold tendencies of this group.

\begin{prop}\label{prop:amen}
Given two infinite groups $G$ and $H$, the following are equivalent.

\begin{enumerate}[(i)]
\item At least one of $G$ or $H$ is amenable.
\item $G\pv H$ admits a faithful doubly transitive amenable action.
\end{enumerate}
\end{prop}

A stronger form of amenability for actions is \textbf{hereditary} amenability, where every orbit of every subgroup is required to be amenable. This can be further strengthened to \textbf{extensive} amenability as defined in~\cite{JMMBS18}.

\begin{prop}\label{prop:extamen}
Given two infinite groups $G$ and $H$, the following are equivalent.

\begin{enumerate}[(i)]
\item Both $G$ and $H$ are amenable.
\item $G\pv H$ admits a faithful doubly transitive extensively amenable action.
\item $G\pv H$ admits a faithful doubly transitive hereditarily amenable action.
\end{enumerate}
\end{prop}

Finally, we turn to an approximation property for $G\pv H$. Following Malcev~\cite[\S7.2]{Malcev_book}, a group is \textbf{locally embeddable into finite groups}, or \textbf{LEF}, if every finite subset of the group can be realised in a finite group with the same multiplication map (where defined). In particular, residually finite groups are LEF. More trivially, so are locally finite groups. This notion was further studied notably by St\"{e}pin~\cite{Stepin83, Stepin84} and Vershik--Gordon~\cite{Vershik-Gordon}; we refer to~\cite{Vershik-Gordon} and to~\cite[Chapter~7]{Ceccherini-Silberstein-Coornaert} for more background. A more general approximation property is \emph{soficity}, as introduced by Gromov~\cite{Gromov_ax} and Weiss~\cite{Weiss00}.

It is easy to see that direct products preserve LEF, and it is known that free products do so too, see Corollary~1.6 in~\cite{Berlai16}. We obtain the following weaker form of permanence.

\begin{thm}\label{thm:LEF}
    If $G$ and $H$ are two residually finite groups, then $G \pv H$ is locally embeddable into finite groups. In particular, it is sofic.
\end{thm}

This turns out to imply a behaviour contrasting with both free and direct products: 

\begin{cor}\label{cor:rf}
If $G$ and $H$ are residually finite and infinite, then $G \pv H$ is not finitely presented.
\end{cor}

We close this introduction with a comment on functoriality. The definition of $G\pv H$ seems rather natural, informally, or at least it is an obvious construct to consider from the viewpoint of pointed permutation groups. It is not, however, natural in the mathematical sense for the category of groups. Weaker statements hold, for instance naturality for monomorphisms of infinite groups. Note also that this ``product'' satisfies an evident commutativity, but fails associativity. We refer to Section~\ref{sec:nat} for all these observations, and to Section~\ref{sec:gen} for generalisations of the construction.

Finally, Section~\ref{sec:q} proposes a few questions.

\subsection*{Acknowledgements}
We are very grateful to Yves Cornulier for his comments on an earlier version of this note; he brought to our attention several references that we were not aware of. Likewise, we thank Pierre de la Harpe warmly for his comments and references.

\setcounter{tocdepth}{1}
\tableofcontents

\section{First properties}

In the entire text, we keep the following conventions. Given two groups $G$ and $H$, we write $\ul G$ for a copy of $G$ viewed as a $G$-set under the left multiplication action (i.e.\ $g \ul g' = \ul{g g'}$), and similarly for $\ul H$. It is part of the definition that the sets $\ul G$ and $\ul H$ are disjoint except for $\ul e$. That is, we write $e$ for the neutral element of any group and this causes no ambiguity since the neutral elements of $G$ and of $H$ are both mapped to a single point $\ul e$ in $\ul G \cup \ul H$. As is customary with products, we emphasise that we consider two ``copies'' of $G$ when writing $G\pv G$.

We identify $G$ with the group of permutations of $\ul G \cup \ul H$ that fixes $\ul H \smallsetminus \{\ul e\}$ pointwise and acts as $G$ on $\ul G$; likewise for $H$. Thus $G$ and $H$ are subgroups of $G\pv H$, which is by definition the group generated by these two permutation groups of $\ul G \cup \ul H$.

We can practise these definitions by showing that $g\in G$ and $h\in H$ freely generate a free semigroup in $G\pv H$ as soon as they are both of infinite order:

\begin{proof}[Proof of Proposition~\ref{prop:pong}]
The argument is a \emph{pong lemma}, the pong ball being $\ul e$. Let $g\in G$ and $h\in H$ be elements of infinite order. Given two ``non-negative words'' $W, W'$ in $g$ and $h$, i.e.\ elements of the free monoid on $\{g, h\}$, we denote by $w, w'$ their evaluations in $G\pv H$. Suppose for a contradiction that $W\neq W'$ but that $w=w'$.

There is no loss of generality in taking a pair $W, W'$ minimising the sum of the lengths of these two words. Since the words cannot both be empty, we can assume that $W$ has a rightmost letter and there is again no loss of generality in supposing that this letter is $g$. It follows that $w \ul e$ lies in $\ul G \smallsetminus\{\ul e\}$ because every (non-empty) right prefix of $W$ will map $\ul e$ to some $\ul{g^n}$ with $n>0$. Therefore, $w' \ul e$ also lies in $\ul G \smallsetminus\{\ul e\}$ and in particular $W'$ is also non-empty. If the rightmost letter of $W'$ were $h$, the same argument would show that $w' \ul e$ lies in $\ul H \smallsetminus\{\ul e\}$. Thus, $W'$ also admits $g$ as its rightmost letter; this contradicts the minimality of the pair $W, W'$.
\end{proof}

Given three distinct elements $x,y,z$ of a set, we recall that $(x;y;z)$ denotes the permutation given by the $3$-cycle
$$x \mapsto y \mapsto z \mapsto x$$
(and fixing the remainder of the set). Our convention for commutators is $[g,h]= g h g\inv h\inv$.

The following computation implies in particular already Proposition~\ref{prop:3tor}.

\begin{prop}\label{prop:tricycle}
If $g\in G$ and $h\in H$ are both non-trivial, then $[g, h] = \tricycle{e}{g}{h}$.
\end{prop}

\begin{proof}
    Let $x \in G \smallsetminus \{e, g\}$. Since $H$ acts trivially on $\ul G \smallsetminus \{\ul e\}$ and $g\inv \ul x \neq \ul e$, we have
\begin{align*}
    [g, h] \ul x    &= g h g\inv h\inv \ul x = g h g\inv \ul x \\
                    &= g h \ul{g\inv x} = g \ul{g\inv x} \\
                    &= \ul x.
\end{align*}
A similar computation shows that $[g, h] \ul x = \ul x$ for any $x \in H \smallsetminus \{e, h\}$. Lastly, on the subset $\{\ul e, \ul g, \ul h\}$, the permutation $[g, h]$ acts as a $3$-cycle:
\begin{align*}
    [g, h] \ul e &= ghg\inv \ul{h\inv} = g h \ul{h\inv} = \ul g, \\
    [g, h] \ul g &= ghg\inv \ul{g} = g h \ul{e} = \ul h, \\
    [g, h] \ul h &= ghg\inv \ul{e} = g h \ul{g\inv} = \ul e.
\end{align*}
\end{proof}

Note that the subgroup $[G, H]$ of $G \pv H$ is normal since $G \pv H$ is generated by $G$ and $H$. The above commutator computation elucidates the subgroup $[G, H]$ and establishes the first part of Proposition~\ref{prop:E}:

\begin{prop}\label{prop:containsaltf}
Suppose that $G$ and $H$ are non-trivial groups.

Then $[G, H] = \Altf (\ul G \cup \ul H)$ holds in $G \pv H$.

In particular, $G \pv H$ contains $\Altf (\ul G \cup \ul H)$ as a normal subgroup. Moreover, this subgroup has trivial centraliser in $G \pv H$ unless $\abs{G} = \abs{H} = 2$.
\end{prop}

\begin{proof}
It is well-known that the group $\Altf(\ul G \cup \ul H)$ is generated by all $3$-cycles. Note on the other hand that the action of $G \pv H$ on $\ul G \cup \ul H$ is transitive by construction. Therefore, any $3$-cycle can be conjugated by an element of $G \pv H$ in such a way that it is of the form $(\ul e; x; y)$. Since $[G, H]$ it normal, we can restrict our attention to such $3$-cycles and show that they are indeed in $[G,  H]$. This already follows from Proposition~\ref{prop:tricycle} when $x$ and $y$ are not both simultaneously in $\ul G$ or in $\ul H$. By symmetry between $G$ and $H$, it therefore only remains to show that $[G, H]$ contains every $3$-cycle $\tricycle{e}{g}{g'}$, where $g$ and $g'$ are distinct non-trivial elements of $G$.

To this end, consider any non-trivial element $h\in H$. Then the element $g\inv h$ of $G \pv H$ maps the triple $(\ul e,\ul g,\ul g')$ to $(\ul h, \ul e, \ul{g\inv g'})$. Therefore it conjugates the $3$-cycle $\tricycle{e}{g}{g'}$ to
$$\tricycle{h}{e}{g\inv g'} = \tricycle{e}{g\inv g'}{h}$$
which we already know to be in $[G, H]$. Thus indeed $[G, H] = \Altf (\ul G \cup \ul H)$.

Finally, since the action of $\Altf (\ul G \cup \ul H)$ on $\ul G \cup \ul H$ is $2$-transitive when $\ul G \cup \ul H$ has at least $4$ elements, its centraliser is trivial unless $\abs{G} = \abs{H} = 2$ (in which case $\Altf (\ul G \cup \ul H)$ is abelian).
\end{proof}

At this point we can completely describe the case of two finite groups.

\begin{thm}\label{thm:finite}
    Let $G$ and $H$ be two non-trivial finite groups. Then
\begin{itemize}
    \item $G \pv H = \Sym (\ul G \cup \ul H)$ if and only if $G$ or $H$ has a nontrivial \emph{cyclic} $2$-Sylow;
    \item $G \pv H = \Alt (\ul G \cup \ul H)$ otherwise.
\end{itemize}
\end{thm}

In particular, $G \pv H$ never surjects onto $G$ nor $H$ when the latter are finite groups of order at least $3$.

As was pointed out to us by P.~de la Harpe, the results of~\cite{Kang04} can be seen as pertaining to random walks on $G \pv H$ with $G$ and $H$ finite.

\begin{proof}[Proof of Theorem~\ref{thm:finite}]
We know from Proposition~\ref{prop:containsaltf} that $G \pv H$ contains $\Altf( \ul G \cup \ul H)$. Since the latter is a proper maximal subgroup of the symmetric group, everything amounts to understanding when $G$ and $H$ both belong to $\Alt(\ul G \cup \ul H)$, or contrariwise when a finite group contains some element whose associated translation is an odd permutation.

If $g \in G$ has order $k \in \NN$, then the translation by $g$ on $\ul {G}$ has $\abs{G} / k$ orbits of length $k$. Therefore its sign as a permutation of $\ul G$ is
\begin{equation*}
    (-1)^{\frac{\abs{G}}{k}(k - 1)},
\end{equation*}
and this is also its sign as permutation of $\ul G \cup \ul H$. The above exponent is always even unless $k$ is even but $\abs{G} / k$ is not, i.e.\ unless there exists an element whose order is even and has the same $2$-valuation as $\abs{G}$. This can only happen if $G$ has a nontrivial $2$-Sylow subgroup which is cyclic.
\end{proof}

Since our definition of $G\pv H$ presents it as a permutation group of the set $\ul G \cup \ul H$; we record the following permutational consequence of Proposition~\ref{prop:containsaltf}.

\begin{cor}\label{cor:high}
Let $G$ and $H$ be non-trivial groups and suppose that at least one is infinite.

Then the action of $G\pv H$ on $\ul G \cup \ul H$ is highly transitive. Moreover, any faithful $2$-transitive $G\pv H$-set is isomorphic, as a $G\pv H$-set, to $\ul G \cup \ul H$.
\end{cor}

We recall here that an action is called \textbf{$n$-transitive} if the induced action on the set of $n$-tuples of distinct points is transitive, and \textbf{highly transitive} if it is $n$-transitive for every $n\in \NN$. The latter is equivalent to the density of the representation to the full permutation group endowed with the usual pointwise topology.

\begin{proof}[Proof of Corollary~\ref{cor:high}]
The high transitivity is due to the fact that $G\pv H$ contains $\Altf( \ul G \cup \ul H)$, which is already highly transitive.

The uniqueness up to isomorphism is a general fact for permutation groups containing $\Altf( \ul G \cup \ul H)$ as a normal subgroup with trivial centraliser (which was recorded in Proposition~\ref{prop:containsaltf}). This general fact is established in Proposition~2.4 of~\cite{LeBoudec-MatteBon_high_arx}.
\end{proof}

\section{The canonical epimorphism and the monolith}

In view of Theorem~\ref{thm:finite}, the main focus of this text is on infinite groups.

\begin{prop}\label{prop:epi}
Let $G$ and $H$ be any groups. If $G$ is infinite, then there is a canonical epimorphism $\pi_G\colon G\pv H \to G$ which is a left inverse for the inclusion $G\to G\pv H$.

More precisely, $\pi_G$ is given by a canonical identification $(G\pv H)/([G,H]H) \cong G$.
\end{prop}

This proposition can of course be applied to both factors when both are infinite. Then $\pi=(\pi_G, \pi_H)$ provides the canonical epimorphism discussed in the introduction, as follows.

\begin{cor}\label{cor:epi}
Let $G$ and $H$ be infinite groups. There is a canonical epimorphism $\pi\colon G \pv H \twoheadrightarrow G \times H$ such that $\pi(g)=(g, e)$ and $\pi(h) = (e,h)$ for all $g\in G$ and $h\in H$. This epimorphism is given by a canonical identification $(G\pv H)/[G, H] \cong G\times H$. \qed
\end{cor}

\begin{proof}[Proof of Proposition~\ref{prop:epi}]
Consider the support $\Supp(\sigma)$ in $\ul G \cup \ul H$ of an element $\sigma$ of the group $[G,H]H$. Since $H$ acts trivially on $\ul G\smallsetminus \{\ul e\}$ and since every element of $[G,H]$ is finitely supported by Proposition~\ref{prop:tricycle}, it follows that $\Supp(\sigma)\cap \ul G$ is finite.

Since on the other hand $\ul G$ is infinite, any two representatives of a coset $\tau [G,H]H$ (where $\tau\in G\pv H$) coincide on a cofinite subset of $\ul G$, and thus coincide there with the multiplication by an element $g\in G$. This defines a map $\tau\mapsto g$ which is a well-defined homomorphism and is the identity on $G<G\pv H$.
\end{proof}

In conclusion, if we combine Proposition~\ref{prop:containsaltf} with Corollary~\ref{cor:epi} and with the fact that $\Altf (\ul{G} \cup \ul{H})$ is a simple group, we obtain the extension~(E) of the introduction (the second part of Proposition~\ref{prop:E}) together with an additional specification:

\begin{prop}\label{prop:monolith}
If $G$ and $H$ are infinite, then the kernel of the canonical epimorphism $G \pv H \twoheadrightarrow G \times H$ is $\Altf (\ul{G} \cup \ul{H})$. Moreover, any nontrivial normal subgroup of $G \pv H$ contains this kernel.
\end{prop}

The latter statement means that $G\pv H$ is a \textbf{monolithic} group, with monolith $\Altf (\ul{G} \cup \ul{H})=[G,H]$. Thus any proper quotient of $G \pv H$ is a quotient of $G \times H$.

\begin{proof}[Proof of Proposition~\ref{prop:monolith}]
It only remains to prove that every non-trivial normal subgroup $N$ of $G\pv H$ contains $\Altf (\ul{G} \cup \ul{H})$. If not, the simplicity of $\Altf (\ul{G} \cup \ul{H})=[G,H]$ implies that it meets $N$ trivially, which entails that these two groups commute. As noted in Proposition~\ref{prop:containsaltf}, this implies that $N$ is trivial.
\end{proof}

\begin{rem}\label{rem:resfin}
The existence of this monolith shows in particular that ${G \pv H}$ is \emph{never} residually finite, and in particular never finitely generated linear ($G$ and $H$ infinite). This stands in contrast to both free and direct products, which preserve residual finiteness and linearity (in equal characteristic). This is clear for direct products; for free products, the former is a theorem of Gruenberg~\cite{Gruenberg57} and the latter of Nisnewitsch~\cite{Nisnewitsch}, see also~\cite{Wehrfritz_free}.
\end{rem}

In conclusion of this section, we have indeed placed $G \pv H$ inbetween the free and the direct product when both $G$ and $H$ are infinite:
$$G * H \lraa G\pv H\lraa G\times H.$$
Moreover, we can write $G\pv H$ as iterated semidirect products
$$G\pv H \cong ([G,H]\rtimes H) \rtimes G  \cong ([G,H]\rtimes G) \rtimes H,$$
but we shall prove next that $G \pv H \twoheadrightarrow G \times H$ itself usually does not split.

\section{The extension usually does not split}

\begin{thm}\label{thm:end}
Let $G,H$ be one-ended finitely generated groups.

Then the canonical epimorphism $G\pv H \twoheadrightarrow G\times H$ does not split.
\end{thm}

\noindent
We recall that one-ended groups are infinite, so that the above projection is indeed defined in view of Corollary~\ref{cor:epi}.

The proof begins with an argument that we borrow from~\cite{Cornulier_near_arx}; after that, we let $G$ and $H$ compete for the insufficient space in $\ul G \cup \ul H$. As mentioned in the introduction, Theorem~\ref{thm:end} can alternatively be deduced from Theorem~7.C.1 in~\cite{Cornulier_near_arx}.

\begin{proof}[Proof of Theorem~\ref{thm:end}]
We suppose for a contradiction that there is a lifting $G\times H\to G\pv H$ and we denote by $\tilde g, \tilde h\in G\pv H$ the images of $g\in G$ and $h\in H$ under this lifting; in particular $\tilde g$ commutes with $\tilde h$.

Consider the Cayley graph of $G$ associated to some finite symmetric generating set $S\se G$. For definiteness, let us choose the left Cayley unoriented simple graph, that is, the edges are all sets $\{g, sg\}$ with $g\in G$, $s\in S$ and $s\neq e$. Let further $\Gamma$ be the graph obtained by deleting every edge $\{g, sg\}$ for which $\tilde s \ul g \neq s \ul g$ (recalling that the right hand side is simply $\ul{sg}$).

In view of the definition of the projection $G\pv H \to G\times H$, we see that for any given $s$, only finitely many edges $\{g, sg\}$ are removed. Since $S$ is finite, we have only removed finitely many edges in the definition of $\Gamma$. By definition of one-endedness, $\Gamma$ has a connected component with finite complement.

Consider the map $\chi\colon G \to \ul G \cup \ul H$ defined by $\chi(g) = \tilde g\inv \ul g$. Then $\chi$ is constant on the connected components of $\Gamma$. Therefore, there is $x\in \ul G \cup \ul H$ such that $\chi(g) = x$ holds outside a finite set of elements $g\in G$. In other words, the set
$$A = \left\{ g\in G : \tilde g x \neq \ul g \right\}$$
is finite. We now consider the $\wt G$-orbit of $x$ in $\ul G \cup \ul H$ and claim that this orbit is regular, i.e.\ with trivial stabilisers. Indeed, suppose $\tilde k x = x$ for some $k\in G$. Since $G$ is infinite, we can choose $g \notin A\cup A k\inv$. Then $\ul g = \tilde g x = \tilde g \tilde k x =\wt{g k} x = \ul{gk}$ and hence $k=e$, as claimed.

It follows that this orbit decomposes into disjoint sets as
$$\wt G x = \wt A x \sqcup  (\ul G\smallsetminus \ul A).$$
In conclusion, $\wt A x$ contains exactly $|A|$ elements and lies in $\ul A \cup \ul H$.

We now apply the same arguments with $G$ and $H$ interchanged, providing $y\in  \ul G \cup \ul H$ and a finite set $B\se H$ with all the corresponding statements.

We further record that since $G$ is finitely generated, there is a finite set $V\se H$ such that $\wt G$ acts trivially on $\ul H\smallsetminus \ul V$. Likewise, $\wt H$ acts trivially on $\ul G\smallsetminus \ul U$ for some finite set $U\se G$.

We claim that the orbits $\wt G x$ and $\wt H y$ are disjoint. Indeed, suppose for a contradiction that $z$ belongs to both. Since $G$ is infinite, we can choose $g\in G$ such that $\tilde g z$ is in $\ul G\smallsetminus (\ul U\cup \{\ul e\})$. Likewise, we can choose $h\in H$ with $\tilde h z$ in $\ul H\smallsetminus (\ul V\cup \{\ul e\})$. Then $\wt{hg}z = \tilde g z \in \ul G\smallsetminus \{\ul e\}$ but this element is also $\wt{gh}z = \tilde h z \in \ul H\smallsetminus \{\ul e\}$, a contradiction confirming the claim.

At this point it follows that $\wt A x$ lies in $\ul A \cup \ul B$, and so does $\wt B y$. The sets $\wt A x$ and $\wt B y$ are disjoint and contain $|A|$, respectively $|B|$, elements. This forces the union $\ul A \cup \ul B$ to be disjoint as well and to coincide with $\wt A x \cup \wt B y$. 

Since $\wt G x$ and $\wt H y$ cannot both contain $\ul e$, we can assume $\ul e \notin \wt G x$, which implies $\ul e \neq \tilde e x$ and thus $e\in A$. Now on the one hand $\ul A \cup \ul B = \wt A x \cup \wt B y$ implies $\ul e \in \wt B y$. But on the other hand, $\ul A$ and $\ul B$ being disjoint forces $e\notin B$, which means $\ul e = \tilde e y = y$. Taken together, $\ul e \in \wt B \ul e$. Since the orbit $\wt H \ul e$ has trivial stabilisers, this shows $e\in B$, a contradiction.
\end{proof}

For infinite groups that are not one-ended, the statement of Theorem~\ref{thm:end} can fail. Since the number of ends is then either two or infinite, the next proposition illustrates both cases according to whether the rank below is $n=1$ or $n\geq 2$.

\begin{prop}
Let $G,H$ be infinite groups.

If one of them is a free group on $n\geq 1$ generators, then the canonical projection $G\pv H \to G\times H$ splits.

\end{prop}

\begin{proof}
Let $F_n$ be a free group on $n\geq 1$ generators. We shall start by defining a $F_n$-action on $F_n$, which we denote by $(g, q)\mapsto \tilde g q$ for $g,q\in F_n$. We further write $\fhi_g$ for the map $F_n\to F_n$ defined by $\fhi_g(q) = \tilde g( g\inv q)$. Since $F_n$ is free, we can specify the action by defining it on a set of free generators. Given a generator $g$, we choose any $g'\neq e,g$ in $F_n$. We define $\tilde g$ by $\tilde g e=e$, $\tilde g g\inv = g'$, $\tilde g (g\inv g') =g$ and $\tilde g q = gq$ in all other cases. Then the map $\fhi_g$ is the cycle $(g; e; g')$ and hence the permutation $\tilde g$ has the following two properties:

\begin{enumerate}
\item $\tilde g e = e$,
\item $\fhi_g$ is a finitely supported permutation of $F_n$ which is even.
\end{enumerate}

Now we observe that these two conditions hold in fact for all $g\in F_n$; this is clear for the first one. As to the second condition, it is inherited from the generators because of the relation
$$\fhi_{ab} = \fhi_a \circ a \circ \fhi_b \circ a\inv$$
which holds for all $a,b\in F_n$.

Turning to the setting of the proposition, suppose now that $G=F_n$ and consider the $G$-action on $\ul G \cup \ul H$ given by the above on $\ul G$ and trivial on $\ul H$; this is well-defined since $\ul e$ is fixed. The second condition shows that this gives a lift $\wt G$ of $G$ in $G\pv H$ because of the construction of the epimorphism $G\pv H\to G$ in the proof of Proposition~\ref{prop:epi}. Since this lift commutes with (the canonical image of) $H$, we have indeed a lift of $G\times H$.
\end{proof}

By construction, each factor $G$ and $H$ in $G\pv H$ lies above the corresponding factor in $G\times H$. In other words, the lifting obstruction of Theorem~\ref{thm:end} really concerns the simultaneous lifting of both factors. Nonetheless, it can be strengthened to hold for the diagonal subgroup when $G=H$:

\begin{thm}\label{thm:end-diago}
Let $G$ be a one-ended finitely generated group.

Then the diagonal subgroup in $G\times G$ cannot be lifted to $G\pv G$.
\end{thm}

Although this statement is of course stronger than the particular case $H=G$ of Theorem~\ref{thm:end}, it will be sufficient to indicate the points where the proof differs from the latter. Again, the statement could also be deduced from results about non-realisability of near actions, specifically using Theorem~7.C.1 in~\cite{Cornulier_near_arx}.

\begin{proof}[Proof of Theorem~\ref{thm:end-diago}]
In order to minimise confusion, we consider two copies $G_1$, $G_2$ of $G$; for each $g\in G$ we write $g_1$ for the corresponding element $(g, e)$ of $G_1\times G_2$ and similarly $g_2=(e, g)$. We suppose for a contradiction that there is a homomorphism $g\mapsto \tilde g$ from $G$ to $G_1 \pv G_2$ such that the image of $\tilde g$ in $G_1\times G_2$ is the diagonal element $g_1 g_2$.

The proof uses the arguments given for Theorem~\ref{thm:end} with minor changes only. We begin with a left Cayley graph for $G$ with respect to a  finite symmetric generating set $S\se G$ and delete every edge $\{g, sg\}$ for which either $\tilde s \ul {g_1} \neq  \ul {s_1 g_1}$ or $\tilde s \ul {g_2} \neq  \ul {s_2 g_2}$ (or both). We have two maps $\chi_i \colon G \to \ul G \cup \ul H$ defined by $\chi_i (g) = \tilde g\inv \ul {g_i}$. The one-end reasoning followed for Theorem~\ref{thm:end} shows that for $i=1,2$ there is $x_i\in \ul{G_1} \cup \ul {G_2}$ (with no indication in which copy of $\ul G$ it lies) such that the set $A_i\subseteq G$ defined by
$$A_i = \left\{ g\in G : \tilde g x_i \neq \ul{g_i}\right\}$$
is finite. We check as above that the $\wt G$-orbits of both $x_i$ are regular. This time the two orbits are disjoint simply because they are orbits of the same group $\wt G$ and cannot coincide since $\wt G x_i$ contains only finitely many points outside $\ul{G_i}$.

At this point we can deduce exactly as for Theorem~\ref{thm:end} that the sets $\wt{A_i} x_i$ are disjoint and lie in $\ul{A_1} \cup \ul {A_2}$, where we wrote simply $\ul{A_i}$ for $\ul{(A_i)_i}$. The end of the proof follows the same strategy, obtaining a contradiction based on the location of $\ul e$ in $\wt{A_1} x_1 \sqcup \wt{A_2} x_2 = \ul{A_1} \sqcup \ul {A_2}$.
\end{proof}

\section{Amenability}\label{sec:amen}
We first record the following basic stability result.

\begin{lem}\label{lem:amen}
If $G$ and $H$ are amenable groups, then so is $G\pv H$.
\end{lem}

\begin{proof}
If $G$ and $H$ are both infinite, then Proposition~\ref{prop:monolith} shows that $G\pv H$ is an extension of two amenable groups: $\Altf(\ul G \cup \ul H)$, which is amenable because it is locally finite, and $G\times H$, which is amenable because both $G$ and $H$ are so.

If $G$ and $H$ are both finite, then so is $G\pv H$ and hence the latter is amenable.

Finally, if exactly one of $G$ or $H$ is infinite, let us assume it is $G$. Then $[G,H]H$ is locally finite by Proposition~\ref{prop:containsaltf} and $G\pv H$ is an extension of $[G,H]H$ by $G$ by Proposition~\ref{prop:epi}.
\end{proof}

We now establish the more surprising statement of Theorem~\ref{thm:co-amen}.

\begin{thm}
Given any two groups $G$ and $H$, the following are equivalent.

\begin{enumerate}[(i)]
\item $G$ is a co-amenable subgroup of $G\pv H$.\label{pt:coamen:coamen}
\item Either $H$ is trivial or both $G$ and $H$ are amenable (hence also $G\pv H$).\label{pt:coamen:both}
\end{enumerate}
\end{thm}

\begin{proof}
\eqref{pt:coamen:coamen}$\Longrightarrow$\eqref{pt:coamen:both}. Suppose that $G$, viewed as a subgroup of $G\pv H$, is co-amenable. We can suppose that $H$ is non-trivial. We can also suppose that $G$ is infinite, since otherwise $G$ is amenable and hence so is $G\pv H$ by co-amenability. Write $A=\Altf(\ul G \cup \ul H)$, so that $A=[G,H]\lhd G\pv H$ by Proposition~\ref{prop:containsaltf}. Since we reduced to the case $G$ infinite, $AG=A\rtimes G$ is semidirect, while $AH$ might not be in case $H$ is finite. Likewise, $G\pv H \cong (AH)\rtimes G$ is semidirect.

We begin with the easy part: the amenability of $H$. The subgroup $A G$ is a fortiori co-amenable in $G\pv H$. Being normal, this means that the quotient group is amenable. If $H$ is infinite, this quotient is $H$ by Proposition~\ref{prop:epi}. If $H$ is finite, it is amenable anyway.

We now establish the amenability of $G$. Consider the map
$$S\colon A H \lra \paf(\ul G),\kern5mm S(\sigma) = \ul G \cap \Supp(\sigma).$$
where $\paf$ denotes the set of finite subsets and $\Supp$ the support of a permutation. For instance, $S(e)=\varnothing$ and $S(h)=\{\ul e\}$ if $h\in H$ is non-trivial. We endow $A H$ with the $G\pv H$-action resulting from viewing $A H$ as the coset space $(G\pv H)/G$. That is, $G$ acts by conjugation and $A H$ by the regular left multiplication. On the other hand, we only endow $\paf(\ul G)$ with its natural $G$-action. Then the map $S$ is $G$-equivariant and thus induces a $G$-equivariant map
$$S_* \colon \means(A H) \lra \means(\paf(\ul G))$$
on the space of means (finitely additive probability measures). Note that $G$ fixes a point in the left hand side, since it even fixes a point in the underlying set $A H$ (namely the identity). By co-amenability, it follows that there is a mean $\mu$ on $A H$ fixed by $G\pv H$. In particular, $S_*(\mu)$ is a $G$-invariant mean on $\paf(\ul G)$.

We claim that $S_*(\mu)\left(\left\{E\in \paf(\ul G): \ul g \in E\right\}\right)=1$ holds for every $g\in G$. Indeed, this number is by definition $\mu\left(\left\{\sigma\in A H : \sigma \ul g \neq \ul g\right\}\right)$. Therefore, the claim amounts to showing that $\mu$ assigns mass zero to the set $\left\{\sigma\in A H : \sigma \ul g = \ul g\right\}$. This set is the stabiliser of $\ul g$ in $A H$ for its action on $\ul G \cup \ul H$. Since this is a transitive action on an infinite set, this stabiliser has infinite index. Now the invariance of $\mu$ under $A H$ implies that the mass is indeed zero because each coset is disjoint with equal mass. This justifies the claim.

Our claim establishes that the $G$-action on $\ul G$ is extensively amenable, cf.\ Definition~1.1 in~\cite{JMMBS18}. In particular, this action is amenable by Lemma~2.1 in~\cite{JMMBS18}. This shows that $G$ is an amenable group.

\medskip
As to the implication \eqref{pt:coamen:both}$\Longrightarrow$\eqref{pt:coamen:coamen}, it is immediate since $G\pv H$ is itself amenable when both $G$ and $H$ are (Lemma~\ref{lem:amen}), and $G\pv 1 = G$.
\end{proof}

The next statement contains Proposition~\ref{prop:amen}.

\begin{prop}\label{prop:amen:high}
Given two infinite groups $G$ and $H$, the following are equivalent.

\begin{enumerate}[(i)]
\item At least one of $G$ or $H$ is amenable.\label{pt:amen:one}
\item $G\pv H$ admits a faithful highly transitive amenable action.\label{pt:amen:amen}
\item $G\pv H$ admits a faithful doubly transitive amenable action.\label{pt:2:amen}
\end{enumerate}
\end{prop}

\begin{proof}
\eqref{pt:amen:one}$\Longrightarrow$\eqref{pt:amen:amen}.
Suppose that $G$ is amenable and consider the action of $G\pv H$ on $\ul G \cup \ul H$. In view of Corollary~\ref{cor:high}, we only need to justify that it is amenable. Since $G$ is amenable, it admits a sequence $(A_n)$ of left F{\o}lner sets $A_n\se G$. Since $G$ is infinite, we can choose $g_n$ such that $A_n g_n$ does not contain $e$. Note that $(A_n g_n)$ is still a left F{\o}lner sequence in $G$; but now $(\ul{A_n g_n})$ is also a F{\o}lner sequence for the $G \pv H$-action on $\ul G \cup \ul H$ because $H$ fixes $\ul{A_n g_n}$.

\eqref{pt:amen:amen}$\Longrightarrow$\eqref{pt:2:amen} is trivial.

\eqref{pt:2:amen}$\Longrightarrow$\eqref{pt:amen:one}.
If $G\pv H$ admits a faithful doubly transitive amenable action, then we can assume by Corollary~\ref{cor:high} that it is the action on $\ul G \cup \ul H$. Let thus $\mu$ be an invariant mean on $\ul G \cup \ul H$. Then $\mu(\ul G)$ and  $\mu(\ul H)$ cannot both vanish since $\mu(\ul G \cup \ul H)=1$. By symmetry we can assume $\mu(\ul G)>0$. After renormalising, we obtain a mean $\mu'$ on $\ul G$ which is still invariant under every element of $G\pv H$ which preserves $\ul G$. In particular, it is invariant under $G$ and witnesses that $G$ is an amenable group.
\end{proof}

We observe that the argument given for \eqref{pt:amen:one}$\Longrightarrow$\eqref{pt:amen:amen} also establishes the following.

\begin{prop}
Let $G$ and $H$ be any groups. If one of $G$ or $H$ is infinite amenable, then $G\pv H$ admits a faithful transitive amenable action.

Moreover, we can take this action to be highly transitive unless the other group is trivial.\qed
\end{prop}

Strengthening amenability to hereditary or extensive amenability flips again the behaviour of $G\pv H$ with respect to $G$ and $H$:

\begin{prop}
Given two infinite groups $G$ and $H$, the following are equivalent.

\begin{enumerate}[(i)]
\item Both $G$ and $H$ are amenable.\label{pt:extamen:both}
\item $G\pv H$ admits a faithful highly transitive extensively amenable action.\label{pt:extamen:ext}
\item $G\pv H$ admits a faithful highly transitive hereditarily amenable action.\label{pt:extamen:her}
\end{enumerate}

\noindent
Moreover, in~\eqref{pt:extamen:ext} and ~\eqref{pt:extamen:her} we can replace high transitivity by double transitivity.
\end{prop}

\begin{proof}
\eqref{pt:extamen:both}$\Longrightarrow$\eqref{pt:extamen:ext}. In view of Corollary~\ref{cor:high}, it suffices to prove that the action of $G\pv H$ on $\ul G \cup \ul H$ is extensively amenable. By Lemma~\ref{lem:amen}, the group $G\pv H$ is amenable. It remains to recall that every action of an amenable group is extensively amenable, see Lemma~2.1 in~\cite{JMMBS18}.

\eqref{pt:extamen:ext}$\Longrightarrow$\eqref{pt:extamen:her} holds by  Corollary~2.3 in~\cite{JMMBS18}.

\eqref{pt:extamen:her}$\Longrightarrow$\eqref{pt:extamen:both}. By Corollary~\ref{cor:high}, the action can be taken to be the canonical $G\pv H$-action on $\ul G \cup \ul H$. By definition of hereditary amenability, the $G$-action on every $G$-orbit in $\ul G \cup \ul H$ remains amenable. Applying this to the $G$-orbit $\ul G$, we deduce that $G$ is an amenable group. We argue likewise for $H$.

\medskip

Finally, the modifications needed for double transitivity in place of high transitivity are taken care of by Corollary~\ref{cor:high} as in the proof of Proposition~\ref{prop:amen:high}.
\end{proof}

\section{Property (FA) and the cubical complex}

\begin{thm}\label{thm:FA}
Let $G$ and $H$ be any groups. Then $G\pv H$ has Serre's property~(FA) if and only if both $G$ and $H$ do.
\end{thm}

\begin{proof}[Proof of Theorem~\ref{thm:FA}]
We start with the main case where both $G$ and $H$ are infinite.

Suppose that $G$ and $H$ have property~(FA) and consider an action of $G\pv H$ by automorphisms on a tree $T$. We write $A=\Altf(\ul G \cup \ul H)$. Upon taking the barycentric subdivision, we can assume that  $G\pv H$ acts without inversions; in particular, if the set $T^A$ of $A$-fixed points is non-empty, then it forms a subtree. In that case the quotient $G\times H$ acts on $T^A$ and hence we find a fixed point of $G\pv H$ since  property~(FA) is closed under finite products by~\cite[\S3.3]{Serre74}. 

We can therefore assume that $A$ has no fixed point in $T$. Since every finitely generated subgroup of $A$ is finite, a compactness argument shows that $A$ fixes a point at infinity $\xi\in \partial T$, see Ex.~2 in~\cite[I\S6.5]{Serre77}. We claim that $\xi$ is the unique such point. Indeed, if $\xi'\neq \xi$ is also fixed, then $A$ preserves the entire geodesic line in $T$ with endpoints $\xi, \xi'$. Being locally finite, $A$ must then fix a vertex on this geodesic, contradicting $T^A\neq \varnothing$, whence the claim.

Since $A$ is normal in $G\pv H$, the uniqueness claim implies that $G\pv H$ fixes $\xi$. Let now $x$ be a vertex fixed by $G$ and $y$ a vertex fixed by $H$. Then $G$ fixes the entire geodesic ray from $x$ to $\xi$ and 
$H$ the geodesic ray from $y$ to $\xi$. Having the same point at infinity, these rays meet. Any point in their intersection is therefore fixed by both $G$ and $H$ and hence by $G\pv H$.

The converse is clear since $G$ and $H$ are quotients of $G\pv H$ by Corollary~\ref{cor:epi}.

\medskip
We consider now the case where exactly one group, say $G$, is infinite. This time, we write $A$ for the kernel of the canonical epimorphism $G\pv H\to G$ given by Proposition~\ref{prop:epi} and we can argue exactly as above. (For the converse direction, a priori only $G$ is a quotient of $G\pv H$, but $H$ has property~(FA) anyways since it is finite.)

Finally, the case where both $G$ and $H$ are finite is trivial since $G\pv H$ is finite as well.
\end{proof}

The construction of the \cat0 cubical complex  consists in applying a classical argument to the action on $\ul G\cup \ul H$, as follows.

\begin{thm}\label{thm:ccc}
Let $G$ and $H$ be infinite groups. Then $G\pv H$ acts by automorphisms on a \cat0 cubical complex $V$ without bounded orbits.

Moreover, each of $G$ and $H$ have fixed vertices in $V$, and there is a unique choice of such vertices that are adjacent.
\end{thm}

In particular, it follows that $G\pv H$ does not have Kazhdan's property~(T), as stated in Theorem~\ref{thm:T}. Indeed, it is well known that an unbounded action on a \cat0 cubical complex gives rise to an unbounded action on a Hilbert space, that is, negates property~(FH), and hence in turn precludes property~(T). This fact has a long history; the most complete reference we know for it is~\cite{Cornulier_FW_arx}.

\begin{proof}[Proof of Theorem~\ref{thm:ccc}]
Consider $\ul G \cup \ul H$ with its canonical $G\pv H$-action. The group $G$ preserves the subset $\ul G$, while the group $H$ preserves $\ul G \smallsetminus\{\ul e\}$. Since $G$ and $H$ generate $G\pv H$, it follows that any $\sigma \in G\pv H$ \textbf{commensurates} the subset $\ul G$, which means by definition that the symmetric difference $\sigma \ul G \triangle \ul G$ is finite. This is a classical setting for the construction of ``walled spaces'' and \cat0 cubical complexes, as initiated notably in~\cite{Sageev95} and~\cite{Haglund-Paulin}; a very general and complete treatment is given in~\cite{Cornulier_FW_arx}. We recall the explicit construction:

Let $V$ be the collection of all subsets $v\se \ul G \cup \ul H$ for which $v\triangle \ul G$ is finite. Consider the (simple, unoriented) graph with vertex set $V$ defined by declaring that $v$ is adjacent to $v'$ whenever $v\triangle v'$ contains exactly one element. The point made in the above references is that this graph is the one-skeleton of a \cat0 cubical complex and the natural $G\pv H$-action on $V$ extends to an action by automorphisms of this complex. By construction, $G$ fixes $\ul G$ when viewed as a vertex, and $H$ fixes the vertex $\ul G \smallsetminus\{\ul e\}$, which is adjacent to $\ul G$.

More generally, a vertex $v\in V$ is fixed by $G$ if and only if $\ul G \se v$; likewise, $v'\in V$ is fixed by $H$ if and only if $v' \se \ul G \smallsetminus \{\ul e\}$. This shows that the pair $(\ul G , \ul G \smallsetminus \{\ul e\})$ is the unique adjacent choice.

Finally, the  $G\pv H$-orbits are unbounded because the combinatorial distance $\left|\sigma \ul G \triangle \ul G\right|$ is unbounded, for instance by high transitivity (Corollary~\ref{cor:high}).
\end{proof}

The action on the above complex $V$ can be described further; the following shows in particular that the orbits of $G\pv H$ in $V$ coincide with the orbits of its normal subgroup $[G,H]$.

\begin{thm}
Consider $s\colon V \to \ZZ$ defined by $s(v) = \left|v\smallsetminus \ul G\right| - \left|\ul G \smallsetminus v\right|$.

\begin{enumerate}[(i)]
\item The map $s$ is a $G\pv H$-invariant surjection.\label{pt:ccc:inv}
\item The fibers $V_n=s\inv(\{n\})$, with $n\in \ZZ$, coincide with the orbits of $G\pv H$ as well as with the orbits of $[G,H]$ in $V$.\label{pt:ccc:orbits}
\item The orbit $V_n$ has a unique $G$-fixed point if $n=0$, infinitely many if $n>0$, and none if $n<0$.\label{pt:ccc:pos}
\item The orbit $V_n$ has a unique $H$-fixed point if $n=-1$, infinitely many if $n<-1$, and none if $n>-1$.\label{pt:ccc:neg}
\end{enumerate}
\end{thm}

\begin{proof}
\eqref{pt:ccc:inv} The map $s$ is onto by the definition of $V$. Note that $s$ is $G$-invariant because the $G$-action preserves both $\left|v\smallsetminus \ul G\right|$ and $\left|\ul G \smallsetminus v\right|$. Consider the following variant $s'$ of $s$:
$$s'(v) = \left|v\smallsetminus (\ul G \smallsetminus \ul \{e\}) \right| - \left|(\ul G \smallsetminus \ul \{e\})  \smallsetminus v\right|.$$
For the same reason as above, $s'$ is $H$-invariant. Therefore, the $G\pv H$-invariance of $s$ will follow if we prove $s'=s+1$. This, however, follows readily by distinguishing the cases where $\ul e$ belong to $v$ or does not. In the former case, $\left|v\smallsetminus (\ul G \smallsetminus \ul \{e\}) \right| = \left|v\smallsetminus \ul G\right| +1$ and $\left|(\ul G \smallsetminus \ul \{e\})  \smallsetminus v\right| = \left|\ul G  \smallsetminus v\right|$. In the latter, $\left|v\smallsetminus (\ul G \smallsetminus \ul \{e\}) \right| = \left|v\smallsetminus \ul G\right|$ and $\left|(\ul G \smallsetminus \ul \{e\})  \smallsetminus v\right| = \left|\ul G  \smallsetminus v\right|-1$.

\medskip
\noindent
\eqref{pt:ccc:orbits} Because of~\eqref{pt:ccc:inv}, it suffices to prove that $[G, H]$ acts transitively on $s\inv(\{n\})$ for any $n\in \ZZ$. We shall use the fact that $[G, H]$ coincides with $\Altf(\ul G \cup \ul H)$ by Proposition~\ref{prop:containsaltf}.

Consider first $n\geq 0$. Choose some $n$-element set $A_n\se \ul H \smallsetminus\{\ul e\}$ and let $v_n =\ul G \cup A_n$, noting $s(v_n) = n$. Consider now any element $v'\in V_n$; we seek $\sigma\in \Altf(\ul G \cup \ul H)$ with $\sigma v' = v_n$. Define the finite sets $B=v'\smallsetminus \ul G$ and $C=\ul G \smallsetminus v'$. Thus $|B| - |C| = n$ and we can choose a subset $B'\se B$ with $|B'|=|C|$, whence also $|B\smallsetminus B'|=n$. The high transitivity of $\Altf(\ul C \cup \ul H)$ on $\ul C \cup \ul H$ implies that there is $\sigma$ in $\Altf(\ul C \cup \ul H)$ satisfying $\sigma(B') = C$ and $\sigma(B\smallsetminus B')=A_n$. We now consider $\sigma$ as an element of $\Altf(\ul G \cup \ul H)$ fixing $\ul G \smallsetminus \ul C$; then indeed $\sigma v' = v_n$ as required.

The case $n<0$ is similar. We choose a $|n|$-element set $A_n\se \ul G$ and let $v_n =\ul G \smallsetminus A_n\in V_n$. Given any $v'\in V_n$, define again $B=v'\smallsetminus \ul G$ and $C=\ul G \smallsetminus v'$. This time there is $C'\se C$ with $|C'|=|B|$ and $|C\smallsetminus C'|=n$ and we find $\sigma$ with $\sigma(C')=B$ and $\sigma(C\smallsetminus C')=A_n$.

\medskip
\noindent
Now that we know that the $G\pv H$-orbits are exactly the sets $V_n$, the points~\eqref{pt:ccc:pos} and~\eqref{pt:ccc:neg} follow from the observation made in the proof of Theorem~\ref{thm:ccc} that $v\in V$ is fixed by $G$ if and only if $\ul G \se v$, and fixed by $H$ if and only if $v \se \ul G \smallsetminus \{\ul e\}$. 
\end{proof}

\begin{rem}
One important difference between the complex $V$ for $G\pv H$ and the Bass--Serre tree of $G*H$ is that vertex stabilisers in $G\pv H$ are much larger than conjugates of $G$ or $H$. For instance, the stabiliser of the vertex $\ul G$ maps onto $G\times H$ under the canonical projection.

Indeed, on the one hand this stabiliser contains $G$. On the other hand, given any non-trivial $h\in H$ we construct a lift $\tilde h\in G\pv H$ as follows. Choose any $h'\neq h, e$ in $H$ and consider the cycle $\sigma = \tricycle{h}{e}{h'}$, which is in $G\pv H$ by Proposition~\ref{prop:containsaltf}. Then the element $\tilde h=\sigma h $ of $G\pv H$ fixes the vertex $\ul G$ (it even fixes every element of $\ul G$).
\end{rem}

\section{Naturality}\label{sec:nat}
The construction of $G\pv H$ cannot be functorial in the usual sense with respect to the factors $G, H$. Specifically, given group homomorphisms $\alpha\colon G\to G'$ and $\beta\colon H\to H'$, we cannot expect to obtain a homomorphism  $G\pv H \to G' \pv H'$ compatible with $\alpha$ and $\beta$ under the canonical inclusions. Indeed, an obstruction arises from the \emph{kernel} of $\alpha$ or $\beta$. This is obvious in the case of finite groups since $G\pv H$ is then typically simple by Theorem~\ref{thm:finite}. Similar examples with infinite groups can readily be given using the fact the $G\pv H$ is monolithic, and a combinatorial obstruction is proposed in Example~\ref{ex:obs} below.

On the other hand, as soon as we exclude kernels, the construction of $G\pv H$ retains as much functoriality as possible. In other words, given subgroups $K<G$ and $L<H$, there is a clear and natural relation between $K\pv L$ and $G\pv H$. This is particularly transparent in the case of infinite groups, where we shall show the following.

\begin{prop}\label{prop:nat-inf}
Consider two groups $G$, $H$ and subgroups $K<G$, $L<H$.

If $K$ and $L$ are infinite, then the canonical embeddings of $K$ and $L$ into $G\pv H$ extend to an embedding $K\pv L \to G\pv H$.
\end{prop}

For finite groups, there is still an additional complication and the above statement does not hold (Example~\ref{exam:nat} below). The naturality can then be expressed in the following weaker form:

\begin{lem}\label{lem:nat-epi}
Consider two groups $G$, $H$ and subgroups $K<G$, $L<H$. Let $\langle K, L \rangle < G\pv H$ be the subgroup generated by the canonical images of $K$ and $L$ in $G\pv H$.

Then there is an epimorphism $\langle K, L \rangle \twoheadrightarrow K\pv L$ compatible with the embeddings of $H$ and $K$.
\end{lem}

In other words, the canonical epimorphism $K*L\twoheadrightarrow K\pv L$ factors through $\langle K, L \rangle$.

\begin{proof}[Proof of Lemma~\ref{lem:nat-epi}]
Consider the subset $\ul K \cup \ul L$ of $\ul G \cup \ul H$. This set is invariant under $K$ and $L$ and hence under $\langle K, L \rangle$. Therefore, we obtain by restriction an action of $\langle K, L \rangle$ on  $\ul K \cup \ul L$ which coincides with the action defining $K\pv L$.
\end{proof}

The fact that this restriction epimorphism is in general not injective explains that it cannot be inverted to give an embedding as in Proposition~\ref{prop:nat-inf} in the case of finite groups:

\begin{exam}\label{exam:nat}
Consider non-trivial finite groups $G$, $H$ and a subgroup $K<G$ with $2<|K| <|G|$. Define $L=H$.

As in the proof of the Lemma~\ref{lem:nat-epi}, $\langle K, L \rangle$ preserves the subset $\ul K \cup \ul L$ of $\ul G \cup \ul H$. Its complement is therefore also invariant under $\langle K, L \rangle$. This complement is $\ul G \smallsetminus \ul K$, which consists of a non-zero number of cosets of $K$ in $G$. The group $L$ acts trivially on this set, while the $K$-action on each coset is regular. Therefore, we obtain an epimorphism $\langle K, L \rangle\twoheadrightarrow K$.

On the other hand, there cannot be any epimorphism $K \pv L\twoheadrightarrow K$ since otherwise Theorem~\ref{thm:finite} would force $K$ to have order at most two.

(A similar obstruction occurs even with $L$ infinite, if the latter has no quotient isomorphic to $K$. Indeed, since $[K,L]$ is the monolith of $K \pv L$, Proposition~\ref{prop:epi} implies that the only proper quotients of $K \pv L$ are the quotients of $L$ and, possibly, a cyclic group of order two.)
\end{exam}

Now in order to prove Proposition~\ref{prop:nat-inf}, it suffices to apply twice the following asymmetric version (where $H$ is allowed to be finite).

\begin{prop}\label{prop:nat-inf-bis}
Consider two groups $G$, $H$ and a subgroup $K<G$.

If $K$ is infinite, then the canonical embeddings of $K$ and $H$ into $G\pv H$ extend to an embedding $K\pv H \to G\pv H$.
\end{prop}

\begin{proof}[Proof of Proposition~\ref{prop:nat-inf-bis}]
If we retain the notation of Lemma~\ref{lem:nat-epi} and its proof (with $L=H$), what we have to show is that the action of $\langle K, H \rangle$ on the subset $\ul K \cup \ul H$ of $\ul G \cup \ul H$ is faithful. Let thus $\sigma\in \langle K, H \rangle$ be an element in the kernel of this action. Consider the morphism $\pi_G\colon G\pv H \to G$  of Proposition~\ref{prop:epi}, which is defined since $G$ is infinite. The construction of $\pi_G$ given in the proof of Proposition~\ref{prop:epi} shows that $\sigma$ acts on all but finitely many points of $\ul G$ as the multiplication by $\pi_G(\sigma)$. Since $\ul K$ is infinite, it follows from the choice of $\sigma$ that $\pi_G(\sigma)$ is trivial. 

What we have to show is that $\sigma$ acts trivially on $\ul G \smallsetminus \ul K$. Since $H$ fixes this set and $K$ preserves it, the action of $\langle K, H \rangle$ there is given by $\pi_G$; this concludes the proof.
\end{proof}

Now that we know that $G\pv H$ behaves well towards embeddings of infinite groups, it makes sense to consider its compatibility with directed unions of groups --- which are nothing but the concrete realisation of inductive limits with injective structure maps.

Recall thus that $\sK$ is a \textbf{directed} family of subgroups $K<G$ of a group $G$ if any two elements of $\sK$ are contained in a further element of $\sK$.

Given Proposition~\ref{prop:nat-inf}, the following statement follows from the fact that $G\pv H$ is generated by $G$ and $H$.

\begin{prop}\label{prop:union}
Suppose that $G$ is the union of a directed family $\sK$ of infinite subgroups and likewise $H$ of a family $\sL$.

Then the embeddings of Proposition~\ref{prop:nat-inf} realise $G\pv H$ as the union of the directed family of subgroups $K\pv L$ with $K\in \sK$, $L\in \sL$.\qed
\end{prop}

\begin{exam}\label{exam:GG_0}
Let $G_0$ be an infinite group and define the sequence $G_n$ by $G_{n+1} = G_n \pv G_0$. We thus obtain a group $G=\lim_n G_n$ satisfying $G \cong G\pv G_0$.
\end{exam}

There is an asymmetric version of Proposition~\ref{prop:union} where $H$ may be finite.

\begin{prop}\label{prop:union:bis}
Suppose that $G$ is the union of a directed family $\sK$ of infinite subgroups and let $H$ be any group.

Then the embeddings of Proposition~\ref{prop:nat-inf-bis} realise $G\pv H$ as the union of the directed family of subgroups $K\pv H$ with $K\in \sK$.\qed
\end{prop}

As mentioned in the introduction, the operation~$\pv$, while it satisfies an obvious commutativity isomorphism, is not associative. We first explain this in the case of infinite groups since this is where $G\pv H$ still resembles a product.

\begin{prop}
Let $G$, $H$ and $J$ be infinite groups. There is no homomorphism
$$f\colon  G \pv \left(H \pv J\right)\lra  \left(G \pv H\right) \pv J$$
compatible with the canonical inclusions of $G$, $H$ and $J$ into each of the two sides.
\end{prop}

\begin{proof}
In the right hand side, the normal closure $N_\mathrm{right}(J)$ of $J$ meets trivially $(G \pv H)$ by Proposition~\ref{prop:epi}. Therefore, the images of $G$ and $H$ in the right hand side generate $(G \pv H)$ also modulo $N_\mathrm{right}(J)$. In particular, non-trivial elements of $G$ and $H$ do not commute in that quotient, see Proposition~\ref{prop:tricycle}. By contrast, in the left hand side, the normal closure $N_\mathrm{left}(J)$ of $J$ contains the monolith $\left[G, (H \pv J)\right]$ which contains $[G,H]$; hence $G$ and $H$ commute in the respective quotient. This shows that $f$ cannot exist.
\end{proof}

For finite groups, the non-associativity can even be detected at the level of the order of the groups:

\begin{exam}
Let $G$, $H$ and $J$ be non-trivial finite groups of order $p$, $q$ and $r$ respectively. In order to avoid distinguishing cases in Theorem~\ref{thm:finite}, assume that the $2$-Sylow subgroups are not cyclic non-trivial (for instance, assume $p,q,r$ odd). Then $G\pv H = \Alt(p+q-1)$ still satisfies this property, because its $2$-Sylow subgroups have solubility length~$\geq 2$ (see Theorem~3 in~\cite{Dmitruk-Sushanskii}). It follows that $(G\pv H)\pv J$ is the alternating group on $\frac12 (p+q-1)! + r-1$ elements, and likewise $G\pv (H\pv J)$ on $\frac12 (r+q-1)! + p-1$ elements. The former number is strictly greater than the latter if and only if $p>r$, which shows that the two products have different orders as soon as $G$ and $J$ do.
\end{exam}

Finally, we return to the lack of naturality with respect to non-injective morphisms with a simple combinatorial observation.

\begin{exam}\label{ex:obs}
Let $G$ and $H$ be any groups. Suppose that $g, g'\in G$ are distinct non-trivial elements and likewise $h, h'\in H$. Recall from Proposition~\ref{prop:tricycle} that the commutator $[g,h]$ in $G\pv H$ is the cycle $\tricycle{e}{g}{h}$. It follows that the element $\sigma= [g,h] [g', h']$ is the cycle $(\ul e; \ul g'; \ul h'; \ul g; \ul h)$ which has order~$5$.

If on the other hand $g=g'\neq e$ (but $h, h'$ are still distinct and non-trivial), then $\sigma$ has order~$2$ because one checks that it is the product of the transpositions $(\ul e; \ul h)$ and $(\ul g; \ul h')$.

If finally $g=g'\neq e$ and $h=h'\neq e$ then Proposition~\ref{prop:tricycle} implies that $\sigma$ has order~$3$.

Since $5$, $3$ and $2$ are coprime, we see that no pair of homomorphisms defined on $G$ and $H$ can be extended to a map on $G\pv H$ that respects these combinatorics unless it is injective (or trivial) on $\{g, g'\}$ and on $\{h, h'\}$.
\end{exam}

\section{Approximations by finite groups}\label{sec:LEF}
We turn our attention to the class of groups locally embeddable into finite groups (LEF) introduced in~\cite{Vershik-Gordon}. We first justify the corollary to Theorem~\ref{thm:LEF}, namely that $G \pv H$ is never finitely presented when $G$ and $H$ are residually finite and infinite.

\begin{proof}[Proof of Corollary~\ref{cor:rf}]
By Theorem~\ref{thm:LEF}, $G \pv H$ is LEF. If it is also finitely presented, then it must be residually finite (Theorem p.~58 in~\cite{Vershik-Gordon}). This is impossible when $G$ and $H$ are infinite, see Remark~\ref{rem:resfin}.
\end{proof}

We now establish the main result of this section.

\begin{proof}[Proof of Theorem~\ref{thm:LEF}]
The main case is when both $G$ and $H$ are infinite countable; we begin with this situation. We endow $G$ and $H$ with proper lengths $\ell_G$ and $\ell_H$ respectively. We recall that this means $\ell_G(g)$ is the distance $d(g,e)$ for some proper left-invariant distance $d$ on $G$, while properness means that all balls are finite. This exists for every countable group; for instance, $\ell_G(g)$ can be taken to be a suitably weighted word length. We write $B_G (n)$ and $B_H (n)$ for the corresponding closed balls of radius $n \in \NN$. Define $C_n \se \ul G \cup \ul H$ to be the set $\ul{B_G (n)} \cup \ul{B_H (n)}$. Let $F_n \se G \pv H$ be the set $B_G (n) B_H (n) \Alt(C_n)$, i.e.
\begin{equation*}
    F_n = \left\{gha \  \middle| \  \ell_G (g) \leq n,\ \ell_H (h) \leq n,\ \Supp(a) \subseteq C_n\right\}.
\end{equation*}

    Observe that $G \pv H = \bigcup_n F_n$ since $G \pv H = G \ltimes ( H \ltimes \Altf(\ul G \cup \ul H))$. Moreover, for any $\sigma \in F_n$, the elements $g \in G, h \in H$ and $a \in \Alt(C_n)$ such that $\sigma = gha$ are uniquely determined.

    We record for later use that, if $\sigma_i = g_i h_i a_i \in F_n$ with $g_i \in B_G (n)$, $h_i \in B_H (n)$ and $a_i \in \Alt(C_n)$ ($i = 1, 2$), then 
\begin{align}\label{eq:sof}
    \sigma_1 \sigma_2  &= g_1 h_1 a_1 g_2 h_2 a_2 \notag \\
        &= g_1 h_1 (g_2 h_2) (g_2 h_2)\inv a_1 g_2 h_2 a_2 \notag \\
        &= (g_1 g_2) h_1 [h_1\inv, g_2\inv] h_2 (g_2 h_2)\inv a_1 g_2 h_2 a_2 \notag \\
        &= (g_1 g_2) (h_1 h_2) \left\{ \left(h_2\inv [h_1\inv, g_2\inv] h_2 \right) \left( (g_2 h_2)\inv a_1 (g_2 h_2) \right) a_2 \right\}.
\end{align}
One the one hand, $g_1 g_2 \in B_G (2n)$ and  $h_1 h_2\in B_H (2n)$. On the other hand, $[h_1\inv, g_2\inv] \in \Alt(C_n)$ by Proposition~\ref{prop:tricycle} and moreover the conjugate of $\Alt(C_n)$ by any element of $B_G (n)$ or of $B_H (n)$ remains in $\Alt(C_{2n})$. Therefore, the calculation~\eqref{eq:sof} shows in particular that $F_n F_n \subseteq F_{2n}$. Our goal is to build, for each $n$, an injective map from $F_{2n}$ into a finite group that is multiplicative on $F_n$.

    Let $G_n$ and $H_n$ be respective finite quotients of $G$ and $H$ such that the associated epimorphisms $\pi_{G_n}\colon G \rightarrow G_n$ and $\pi_{H_n}\colon H \rightarrow H_n$ are injective on $B_G (4n)$ and on $B_H (4n)$. Let $\pi\colon \ul G \cup \ul H \rightarrow \ul{G_n} \cup \ul{H_n}$ be the (set-theoretic) surjection defined by $\pi(\ul g) = \ul{\pi_{G_n} (g)}$ and $\pi(\ul h) = \ul{\pi_{H_n} (h)}$ for $g \in G$ and $h \in H$. Observe that, by construction, the map $\pi$ enjoys a weak form of equivariance in the following sense:
\begin{equation}\label{eq:sofequiv}
\begin{aligned}
    \pi (g z)  &= \pi_{G_n} (g) \pi(z) && \text{for any } g \in G \text{ and } z \in  \ul G \cup \left(\ul H \smallsetminus \ul{\ker \pi_{H_n}} \right), \\
    \pi (h z)  &= \pi_{H_n} (h) \pi(z) && \text{for any } h \in H \text{ and } z \in \left( \ul G \smallsetminus \ul{\ker \pi_{G_n}} \right) \cup \ul H.
\end{aligned}
\end{equation}
Indeed, if $z \in \ul G$, then $gz \in \ul G$ and $\pi(gz) = \pi_{G_n} (gz) = \pi_{G_n} (g) \pi_{G_n} (z) = \pi_{G_n} (g) \pi (z)$. If on the other hand $z \in \ul H \smallsetminus \ul{\ker \pi_{H_n}}$, then $gz = z$ and $\pi(z) = \pi_{H_n} (z) \in \ul{H_n} \smallsetminus \{ \ul e \}$, hence $\pi(gz) = \pi(z) = \pi_{G_n} (g) \pi(z)$. Similar computations hold for $H$.

 We write $C'_k = \pi(C_k)$, observing that, by our choices of $G_n$ and $H_n$, the map $\pi$ induces a bijection between $C_k$ and $C'_k$ for any $k \leq 4n$.

    Let finally $\Phi_n\colon F_{2n} \rightarrow G_n \pv H_n$ be defined by $\Phi_n (gha) = \Phi_n(g) \Phi_n (h) \Phi_n (a)$, where
\begin{align*}
    \Phi_n (g)  &= \pi_{G_n} (g) && \text{for } g \in G, \\
    \Phi_n (h)  &= \pi_{H_n} (h) && \text{for } h \in H, \\
    \Phi_n (a)(x)  &= \begin{cases} \pi (a (y)) & \text{if } x \in C'_{2n}, \text{where } y \in C_{2n}, \pi(y) = x \\ x & \text{if } x \not\in C'_{2n} \end{cases} && \text{for } a \in \Alt(C_{2n}).
\end{align*}
Observe that $\Phi_n (a)$ is indeed well-defined since $a$ preserves $C_{2n}$ and $\pi$ induces a bijection between $C_{2n}$ and $C'_{2n}$. Moreover, the relation
\begin{equation}\label{eq:sofphialt}
    \Phi_n (a) (\pi(y)) = \pi (a (y))
\end{equation}
holds more generally for any $y \in C_{4n}$. Indeed, if $y \in C_{4n} \smallsetminus C_{2n}$, then $\pi(y) \in C'_{4n} \smallsetminus C'_{2n}$, hence $\Phi_n (a)$ fixes $\pi(y)$. But $a$ fixes $y$ also since $\Supp (a) \subseteq C_{2n}$.

\bigskip

Let us first prove that $\Phi_n$ is injective. Let $gha$ and $g'h'a'$ in $F_{2n}$ be such that $\Phi_n (gha) = \Phi_n (g'h'a')$. Applying the equality
\begin{equation*}
    \Phi_n (g) \Phi_n (h) \Phi_n (a) x = \Phi_n (g') \Phi_n (h') \Phi_n (a') x \qquad \qquad (x \in \ul{G_n} \cup \ul {H_n})
\end{equation*}
first to some $x \in \ul {G_n} \smallsetminus C'_{2n}$ then to some $x \in \ul {H_n} \smallsetminus C'_{2n}$ yields successively
\begin{equation*}
    \Phi_n (g) = \Phi_n (g') \qquad \text{and} \qquad \Phi_n (h) = \Phi_n (h'),
\end{equation*}
since $\Phi_n (a)$ and $\Phi_n (a')$ have support in $C'_{2n}$ and since $G_n$ (resp.~$H_n$) acts regularly on $\ul {G_n}$ (resp.~on $\ul {H_n}$) and trivially on $\ul {H_n} \smallsetminus \{\ul e \}$ (resp.~on $\ul {G_n} \smallsetminus \{\ul e \}$). As $\pi_{G_n}$ and $\pi_{H_n}$ are injective on $B_G (2n)$ and on $B_H (2n)$, respectively, we conclude that $g = g'$ and $h = h'$. Lastly, the equality
\begin{equation*}
    \Phi_n (a) x = \Phi_n (a') x \qquad \qquad (x \in \ul{G_n} \cup \ul {H_n})
\end{equation*}
becomes
\begin{equation*}
    \pi(a(y)) = \pi(a'(y))
\end{equation*}
for any $y \in C_{2n}$, hence $a = a'$ since $\pi$ is injective on $C_{2n}$, which is the support of both $a$ and $a'$. Hence $\Phi_n$ is indeed injective.

Our goal is now to prove that $\Phi_n$ is multiplicative on $F_n$, i.e.~that if $\sigma_1, \sigma_2 \in F_n$, then $\Phi_n (\sigma_1 \sigma_2) = \Phi_n (\sigma_1) \Phi_n (\sigma_2)$. Obviously, $\Phi_n$ is multiplicative on $B_G (n)$, on $B_H (n)$ and on $\Alt (C_n)$. Hence, in view of Eq.~\eqref{eq:sof}, we are left to prove the following equality:
\begin{multline}\label{eq:sofmult}
    \Phi_n \left(h_2\inv [h_1\inv, g_2\inv] h_2 \right) \Phi_n \left( (g_2 h_2)\inv a_1 (g_2 h_2) \right) = \\ \left(\Phi_n(h_2)\inv [\Phi_n (h_1)\inv, \Phi_n (g_2)\inv] \Phi_n (h_2) \right) \left( (\Phi_n (g_2 ) \Phi_n (h_2))\inv \Phi_n (a_1) (\Phi_n (g_2) \Phi_n(h_2) ) \right)
\end{multline}

We begin by the first term of each product. The permutation $h_2\inv [h_1\inv, g_2\inv] h_2$ is the $3$-cycle $\tricycle{h_2\inv}{h_2\inv h_1\inv}{g_2\inv}$. Each point of this cycle belongs to $C_{2n}$, hence
\begin{equation*}
    \Phi_n (h_2\inv [h_1\inv, g_2\inv] h_2) = \tricycle{\pi_{H_n} (h_2\inv)}{\pi_{H_n} (h_2\inv h_1\inv)}{\pi_{G_n}(g_2\inv)}.
\end{equation*}
On the other hand,
\begin{align*}
    \Phi_n(h_2)\inv [\Phi_n (h_1)\inv, \Phi_n (g_2)\inv] \Phi_n (h_2) &= \pi_{H_n}(h_2)\inv [\pi_{H_n} (h_1)\inv, \pi_{G_n} (g_2)\inv] \pi_{H_n} (h_2) \\
        &= (\ul{\pi_{H_n} (h_2\inv)} ; \pi_{H_n} (h_2\inv) \ul{\pi_{H_n}(h_1\inv)} ; \ul{\pi_{G_n}(g_2\inv)}).
\end{align*}
But thanks to the equivariance~\eqref{eq:sofequiv}, we have
\begin{equation*}
    \pi_{H_n} (h_2\inv) \ul{\pi_{H_n}(h_1\inv)} = \pi_{H_n} (h_2\inv) \pi (h_1\inv) = \pi (h_2\inv h_1\inv) = \ul{\pi_{H_n} (h_2\inv h_1 \inv)},
\end{equation*}
hence
\begin{equation*}
    \Phi_n \left(h_2\inv [h_1\inv, g_2\inv] h_2 \right) = \Phi_n(h_2)\inv [\Phi_n (h_1)\inv, \Phi_n (g_2)\inv] \Phi_n (h_2),
\end{equation*}
as needed.

We now turn our attention to the second terms of each side of \eqref{eq:sofmult}. Let $x \in C'_{2n}$ and $y \in C_{2n}$ such that $\pi(y) = x$. 

Since, on the one hand, $C_{4n} \se \left( \ul G \smallsetminus \ul{\ker \pi_{G_n}} \right) \cup \left(\ul H \smallsetminus \ul{\ker \pi_{H_n}} \right) \cup \{ \ul e\}$ and, on the other hand, $g_2 h_2 C_{2n} \se C_{3n}$ and $(g_2 h_2)\inv C_{3n} \se C_{4n}$ we can use the equivariance relations~\eqref{eq:sofequiv} and the fact that $a_1$ preserves $C_{3n}$ (or indeed any superset of $C_n$) in order to compute:
\begin{align*}
      \Phi_n \left( (g_2 h_2)\inv a_1 (g_2 h_2) \right)(x)  &= \pi \left( \left\{ (g_2 h_2)\inv a_1 (g_2 h_2) \right\}(y) \right) \\
        &= \pi_{H_n} (h_2)\inv \pi_{G_n}(g_2)\inv \pi(a_1 (g_2 h_2 y)). \\
\intertext{As $g_2 h_2 y \in C_{3n}$, we can use the relation~\eqref{eq:sofphialt} and, once again, the equivariance~\eqref{eq:sofequiv}, to conclude the computations:}
      \dots  &= \pi_{H_n} (h_2)\inv \pi_{G_n}(g_2)\inv \Phi_n(a_1) (\pi_{G_n} (g_2) \pi_{H_n} (h_2) x) \\
        &= \left( (\Phi_n (g_2 ) \Phi_n (h_2))\inv \Phi_n (a_1) (\Phi_n (g_2) \Phi_n(h_2) ) \right) (x).
\end{align*}

If now $x \not\in C'_{2n}$, then $\Phi_n (g_2) \Phi_n (h_2) x \not\in C'_n$. Either this points belongs to $C'_{2n} \smallsetminus C'_n$ or it does not. In both cases it is fixed by $\Phi_n (a_1)$ since the support of $a_1$ is included in $C_n$. Hence
\begin{align*}
        \left( (\Phi_n (g_2 ) \Phi_n (h_2))\inv \Phi_n (a_1) (\Phi_n (g_2) \Phi_n(h_2) ) \right) (x) &= (\Phi_n (g_2 ) \Phi_n (h_2))\inv (\Phi_n (g_2) \Phi_n(h_2) ) (x) \\
        &= x \\
        &= \Phi_n \left( (g_2 h_2)\inv a_1 (g_2 h_2) \right)(x) 
\end{align*}
(the last line being due to $x \not\in C'_{2n}$).

To sum up, we have checked that
\begin{equation*}
        (\Phi_n (g_2 ) \Phi_n (h_2))\inv \Phi_n (a_1) (\Phi_n (g_2) \Phi_n(h_2)) = \Phi_n \left( (g_2 h_2)\inv a_1 (g_2 h_2) \right)
\end{equation*}
and hence that the equality~\eqref{eq:sofmult} holds, namely that $\Phi_n$ is indeed multiplicative on $F_n$. This concludes the proof that $G \pv H$ is LEF when the residually finite groups $G, H$ are both infinite countable.

\medskip
We now extend the statement to the case where $G, H$ are both infinite. The family $\sK$ of all infinite countable subgroups $K<G$ is directed and covers $G$. The corresponding fact holds for the family $\sL$ of infinite countable subgroups $L<H$. Proposition~\ref{prop:union} therefore realises $G\pv H$ as the directed union of all $K\pv L$, which are LEF by the first case. On the other hand, the LEF property passes to directed unions by its very definition.

\medskip
It remains to consider the case where at least one of $G$ or $H$ is finite. If both are finite, then so is $G\pv H$ and LEF holds trivially. We can therefore assume that $G$ is infinite and $H$ finite. As before, we cover $G$ by all its infinite countable subgroups and apply now Proposition~\ref{prop:union:bis}; this reduces us to the case where $G$ is infinite countable and $H$ finite. Now the argument is a simpler variant of the argument given for the main case above, and becomes very close to the argument given for Lemma~3.2 in~\cite{Elek-Szabo06}. Specifically, the above computations still hold with the following minor modifications:
\begin{itemize}
    \item $C_n$ is now the set $\ul {B_G (n)} \cup \ul H$.
    \item $F_n$ is now $B_G (n) \Alt (C_n)$, or $B_G (n) \Sym(C_n)$ if $H$ has a nontrivial cyclic $2$-Sylow. This distinction is made to ensure that $\bigcup_n F_n = G \pv H$, since the latter is now $G \ltimes ([G, H]H)$ and $[G,H]H$ is $\Altf(\ul G \cup \ul H)$ or $\Symf(\ul G \cup \ul H)$ depending on whether $H$ admits a nontrivial cyclic $2$-Sylow (compare with Theorem~\ref{thm:finite}).
    \item $G_n$ is still chosen so that the projection map $\pi_{G_n}$ is injective on $B_G (4n)$ but $H_n$ is simply $H$.
    \item The surjection $\pi$ is defined as above (where $\pi_{H_n}$ is thus the identity) and $C'_k$ is again $\pi(C_k)$.
    \item The map $\Phi_n$ is defined similarly on $B_G (2n)$ and on $\Alt (C_{2n})$ (or $\Sym (C_{2n})$).
    \item The relation that needs to be proved in order to establish that $\Phi_n$ is multiplicative on $F_n$ is now:
\begin{equation*}
    \Phi_n (g_1 g_2) \Phi_n (g_2\inv a_1 g_2 a_1) = \Phi_n (g_1) \Phi_n (g_2) \left( \Phi_n (g_2)\inv \Phi_n (a_1) \Phi_n (g_2) \Phi_n (a_1)\right).
\end{equation*}
Since $\Phi_n$ is again multiplicative on $B_G (n)$ and on $\Alt (C_n)$ (or $\Sym(C_n)$), this equation reduces to
\begin{equation*}
    \Phi_n (g_2\inv a_1 g_2) = \left( \Phi_n (g_2)\inv \Phi_n (a_1) \Phi_n (g_2)\right).
\end{equation*}
The latter can be proved, as for the second terms of~\eqref{eq:sofmult}, by applying each side to some $x \in \ul {G_n} \cup \ul {H_n}$ and distinguishing between the cases $x \in C'_{2n}$ and $x \not\in C'_{2n}$.
\end{itemize}
\end{proof}

\section{Topologies, completions and generalisations}\label{sec:gen}
Endow the group $\Sym(\ul G \cup \ul H)$ with the topology of pointwise convergence, which is Polish if $G$ and $H$ are countable. Since the group $G \pv H$ acts highly transitively on $\ul G \cup \ul H$ (unless both $G$ and $H$ are finite), it is a dense subgroup of $\Sym(\ul G \cup \ul H)$. An equivalent reformulation is to consider the topology of pointwise convergence on $G \pv H$; then the statement is that the topological group $G \pv H$ admits $\Sym(\ul G \cup \ul H)$ as its (upper) completion.

We shall introduce two other completions $K_G$ and $K_H$ of $G \pv H$, which are (usually non-discrete) Polish groups when $G$ and $H$ are countable, but with the additional property that the diagonal embedding of  $G \pv H$ into $K_G \times K_H$ is \emph{discrete}.

Thus, in a very loose sense, the situation is analogous to an arithmetic group (instead of $G \pv H$) with its diagonal embedding into the product of all local completion at all places (instead of $K_G$ and $K_H$).

\medskip
To set up the notation, let $X$ be any set. We recall that a \emph{bornology} on $X$ is a family $\Born \subseteq \Power (X)$ that is closed under subsets and finite unions. We define $\Sym_\Born (X)$ as the subgroup of $\Sym(X)$ that preserves the bornology $\Born$, i.e. $g \in \Sym_\Born (X)$ if 
\begin{equation*}
    B \in \Born \quad \Longleftrightarrow \quad g(B) \in \Born.
\end{equation*}
We endow $\Sym_\Born$ with the \emph{topology of uniform discrete convergence on $\Born$}, an identity neighbourhood basis of which is given by the family of pointwise fixators $\Fix(B)$ for $B \in \Born$. For instance, if $\Born$ is the family of all finite sets, then $\Sym_\Born (X) = \Sym (X)$ and we recover the usual topology of pointwise convergence.

It can easily be checked by hand (and it follows from the more general Theorem~2.2 in~\cite{Gheysens_ULB_arx}) that the group $\Sym_\Born (X)$ is complete for the upper uniform structure. In particular, the group $\Sym_\Born (X)$ is \emph{completely metrisable} if $\Born$ is \emph{countably generated}, that is if there exists a countable subfamily $\Born' \subseteq \Born$ such that any $B \in \Born$ is a subset of some $B' \in \Born'$. 

\medskip

Coming back to $G \pv H$, we observe that this group preserves two natural bornologies, namely the family $\Born_G$ of all subsets $B \se \ul G \cup \ul H$ such that $B \cap \ul G$ is finite, and the analogous family $\Born_H$ of all subsets $B$ such that $B \cap \ul H$ is finite. (There is a third invariant bornology, $\Born_G \cap \Born_H$, which is nothing more than the family of finite subsets.) Let $K_G$ (resp.~$K_H$) be the closure of $G \pv H$ in the topological group $\Sym_{\Born_G} (\ul G \cup \ul H)$ (resp.~$\Sym_{\Born_H} (\ul G \cup \ul H)$). By the above discussion, $K_G$ and $K_H$ are Polish if $G$ and $H$ are countable.

\begin{prop}
    The group $K_G$ is non-discrete if (and only if) $G$ is infinite.
\end{prop}

\begin{proof}
 If $G$ is finite, then $\Born_G = \Power (\ul G \cup \ul H)$ and hence $\Sym_{\Born_G} (\ul G \cup \ul H)$ is the whole group $\Sym (\ul G \cup \ul H)$ endowed with the discrete topology.

    For the other direction, it suffices to exhibit a permutation $\varphi \in K_G$ not in $G \pv H$. Since $G$ is infinite, we can partition it into two infinite subsets $G_2 \sqcup G_3$ and then choose further partitions of $G_2$ into infinitely many pairs and of $G_3$ into infinitely many triples. Let $\varphi$ be the permutation that acts as a transposition on each pair of $\ul{G_2}$, as a $3$-cycle on each triple of $\ul{G_3}$ and as the identity on $\ul H \smallsetminus \{ \ul e \}$. Then, for any $B \in \Born_G$, there is an element of $\Altf (\ul G \cup \ul H)$ that agrees with $\varphi$ on $B$, hence $\varphi$ belongs to $K_G$. But $\varphi$ cannot be an element of $G \pv H$, because there is no cofinite subset of $\ul G$ on which $\varphi$ coincides with the multiplication by an element of $G$, since there are orbits of different sizes.
\end{proof}

\begin{prop}
    If $G$ and $H$ are infinite, then the diagonal embedding $G \pv H \rightarrow K_G \times K_H$ has a discrete image.
\end{prop}

\begin{proof}
    Indeed, $\ul H \in \Born_G$ and $\ul G \in \Born_H$, hence $\Fix \ul H \times \Fix \ul G$ is an identity neighbourhood of $K_G \times K_H$. This neighbourhood meets the diagonal only in the element $(e, e)$, so the image of $G \pv H$ is discrete.
\end{proof}

\medskip
Finally, we indicate two natural generalisations of the construction $G\pv H$.

\medskip
A first observation is that we can extend the definition to any family $\{G_i\}_{i\in I}$ of groups $G_i$ indexed by an arbitrary set $I$. To this end, we consider again for each $i\in I$ a copy $\ul{G_i}$ of the set underlying $G_i$, take the disjoint union of all these sets and then identify all the neutral elements $\ul e\in\ul{G_i}$. We consider each $G_i$ as a permutation group of the set $X=\bigcup_{i\in I} \ul{G_i}$ by giving it the regular action on $\ul{G_i}$ and the trivial action on the complement $X\smallsetminus \ul{G_i}$. Finally, we define $\bigpv_{i\in I} G_i$ to be the permutation group of $X$ generated by these copies of each $G_i$.

Thus $G_1\pv G_2$ is the special case $\bigpv_{i\in\{1,2\}} G_i$ but we should take care to distinguish $\bigpv_{i\in\{1,2,3\}} G_i$ from the groups $G_1\pv (G_2\pv G_3)$ and $(G_1\pv G_2) \pv G_3$ considered in Section~\ref{sec:nat}. A good number of the results presented in this text admit a straightforward adaptation to this setting.

\medskip

A second generalisation is based on the observation that $G\pv H$ is not merely a group, but comes with a canonical incarnation as a permutation group (on the set $\ul G \cup \ul H$). Moreover, there is a natural choice of a point in this set, namely $\ul e$.

At this juncture we recall that a group $G$ can be considered as a special case of a pointed $G$-set in the following more precise sense. The functor associating to a group $G$ the pointed $G$-set $(\ul G, \ul e)$ is a fully faithful embedding of the category of groups into the category of pointed sets with a group action; this would not be the case for unpointed sets.

To extend the definition of $G\pv H$, we can start more generally with a pointed $G$-set $(X, *)$ and a pointed $H$-set $(Y, *)$. We then form the disjoint union of $X$ and $Y$ but with both copies of $*$ identified and still denote the resulting pointed space as $(X\cup Y, *)$. We consider now the group of permutations of $X\cup Y$ generated by the given images of $G$ and $H$. Some of the results presented in this text extend to this setting, and more generally we can form the pointed set with group action associated to an arbitrary family of groups $G_i$ given each with its pointed $G_i$-set $(X_i, *)$.

\medskip
We note that all these generalisations also support natural topologies. The obvious one is the topology of pointwise convergence inherited from the ambient symmetric group, but as before the various ``factors'' define bornologies $\Born$ for which we can consider the topology of uniform discrete convergence on $\Born$.

\section{Questions}\label{sec:q}

\begin{question}\label{q:GG_0}
Is there a \emph{finitely generated} group $G$ satisfying $G \cong G\pv G_0$ for a non-trivial group $G_0$?
\end{question}

The corresponding question has a non-trivial positive answer in the case of direct products, see Proposition~1 in~\cite{Tyrer-Jones}, and remains open for finitely presented groups~\cite{Hirshon02}. It has a negative answer for free products because of Grushko's theorem~\cite{Grushko}.

If no finite generation is assumed, we have seen in Example~\ref{exam:GG_0} that a straightforward limit argument provides $G\cong G\pv G_0$.

\begin{question}
Does $G\pv H$ have Haagerup's property when both $G$ and $H$ do?
\end{question}

If the answer is positive, then we would expect it to be established with the method of \emph{spaces with measured walls}, compare~\cite{Cherix-Martin-Valette,Cornulier-Stalder-Valette12,Robertson-Steger98}. In that case, we would also expect the corresponding positive answer for the property of admitting a proper action on a \cat0 cube complex (Property~PW in~\cite{Cornulier-Stalder-Valette08}).

\begin{question}\label{q:fp}
Can $G\pv H$ ever be finitely presented when $G$ and $H$ are infinite?
\end{question}

Since $G \pv H$ is monolithic, finite presentability is equivalent to being isolated in the space of marked groups, see Proposition~2 in~\cite{Cornulier-Guyot-Pitsch2007}. As established in Corollary~\ref{cor:rf}, $G\pv H$ is not finitely presented when $G$ and $H$ are infinite and residually finite.

\medskip
Section~\ref{sec:amen} explores several aspects of amenability but we have no complete answer to the following.

\begin{prob}
Characterise the pairs $G$, $H$ such that $G\pv H$ admits a faithful transitive amenable action.
\end{prob}

Recall that a satisfactory characterisation in the case of free products is obtained in~\cite{Glasner-Monod}, while the case of direct products is elementary.

\medskip
In view of Theorem~\ref{thm:LEF}, it is natural to ask:
\begin{question}
    Is $G \pv H$ sofic whenever $G$ and $H$ are so?
\end{question}
Recall that direct and free products indeed preserve soficity (Theorem~1 in~\cite{Elek-Szabo06}).


\bibliographystyle{plain}
\bibliography{biblio_prod}

\end{document}